\pdfoutput=1
\documentclass[11pt]{article}

\usepackage[utf8]{inputenc}
\usepackage[pdftex]{graphicx,xcolor}
\usepackage{array,fullpage,enumitem,microtype}
\usepackage[leqno]{amsmath}
\usepackage{amssymb,amsthm,mathtools,bm,stmaryrd,tikz-cd,mathdots}
\usepackage[pdftex,colorlinks=true,allcolors=blue,unicode,bookmarksnumbered,psdextra]{hyperref}
\usepackage{bookmark,tocbibind}
\usepackage[capitalize,nosort]{cleveref}

\usepackage{mymacros}

\let\defn\textbf

\newcommand*\Comgr{\mathrm{Comgr}}
\DeclareMathOperator\Ext{Ext}

\begin{document}

\title{On the Pettis--Johnstone theorem for localic groups}
\author{Ruiyuan Chen}
\date{}
\maketitle

\begin{abstract}
We explain how Johnstone's 1989 proof of the closed subgroup theorem for localic groups can be viewed as a point-free version of Pettis's theorem for Baire topological groups.
We then use it to derive localic versions of the open mapping theorem and automatic continuity of Borel homomorphisms, as well as the non-existence of binary coproducts of complete Boolean algebras.
\end{abstract}

\section{Introduction}

Pettis's theorem is a fundamental result for topological groups, which states, in its usual formulation, that the difference set $A \cdot A^{-1}$ of a non-meager set $A$ with the property of Baire is a neighborhood of the identity.
Many other important applications of Baire category to topological groups and vector spaces are easy consequences of Pettis's theorem, such as the open mapping theorem, automatic continuity, and closedness of embedded subgroups.
For a standard reference, see \cite[\S2.3]{Gidst}.

A locale is a ``formal'' space without an underlying set of points, consisting only of an abstract lattice of ``open sets''.
A motivating theme throughout locale theory is that forgetting about points results in a less pathological counterpart of topological spaces, and this is well-illustrated by the ``Baire category theorem'' for locales, due to Isbell (see e.g., \cite[C1.2.6]{Jeleph}): the intersection of \emph{all} dense open sets in a locale forms a dense sublocale.
The central role of Baire category in topological group theory then suggests analogs of various classical results, such as the aforementioned, for localic groups.
One such analog is given by the ``closed subgroup theorem'' of Isbell--Kříž--Pultr--Rosický \cite{IKPRgrp}, which states that every localic subgroup of a localic group is closed.
While the original proof of the closed subgroup theorem was fairly hands-on lattice-theoretic, Johnstone \cite{Jgpd} gave a more conceptual proof that essentially amounts to a point-free version of Pettis's theorem.
However, to our knowledge, this connection does not appear to be very widely known or used.

In this note, after reviewing some background in \cref{sec:loc}, we give in \cref{sec:pettis} an exposition of Johnstone's proof that makes its analogy to Pettis's theorem completely transparent.
We then derive several consequences of this localic version of Pettis's theorem, which we call the ``Pettis--Johnstone theorem''.
In \cref{sec:homom}, we derive consequences concerning the structure of localic group homomorphisms, including localic versions of the aforementioned results: the open mapping theorem, and automatic continuity of ``$\infty$-Borel group homomorphisms''.
In \cref{sec:cbool}, we derive a purely algebraic consequence having \emph{a priori} nothing to do with localic groups: we show that the coproduct of two (set-sized) complete Boolean algebras may not exist, in that it may be a proper class.

\paragraph*{Acknowledgments}

Research partially supported by NSF grant DMS-2054508.

\section{Preliminaries}
\label{sec:loc}

For standard background on locales, see \cite{Jstone}, \cite[C1.1--2]{Jeleph}, or \cite{PPloc}.
We denote the category of locales by $\!{Loc}$.
In this paper, we are not concerned with matters of constructivity, and work throughout in classical logic (unlike \cite{Jgpd}).

We follow the convention (used in e.g., \cite{Jeleph}) that the ``algebraic'' and ``geometric'' views of locales are to be strictly distinguished, both terminologically and notationally.
Thus, we speak of locales $X$ versus their frames of open sets $\@O(X)$, and continuous maps $f : X -> Y$ versus frame homomorphisms $f^* : \@O(Y) -> \@O(X)$.
We are willing to call elements of $\@O(X)$ \emph{open sets} $U \subseteq X$ (rather than terms like \emph{open parts} or \emph{opens} favored by some authors), because we will later be dealing with more general types of ``sets'', which would be rather awkward to similarly rename; the clear distinction between $X$ versus $\@O(X)$ should remove any potential confusion in the word ``set''.
Likewise, we freely use set-theoretic notation in place of lattice-theoretic ones, e.g.,
\begin{align*}
U \subseteq V  &\coloniff  U \le V, &
U \cap V &:= U \wedge V, &
X &:= \top, &
\emptyset &:= \bot.
\end{align*}
We sometimes speak of ``elements'' $x \in X$ of a locale; this will always mean points, i.e., continuous maps $x : \*1 -> X$, as opposed to open sets $U \in \@O(X)$.

A \defn{product locale} $X \times Y$ is given by a coproduct frame $\@O(X \times Y) := \@O(X) \otimes \@O(Y)$; we denote the ``open rectangles'' by $U \times V \in \@O(X \times Y)$ for $U \in \@O(X)$ and $V \in \@O(Y)$ (rather than the more common $U \otimes V$).
A \defn{sublocale} $Y \subseteq X$ is given by a quotient frame $\@O(X) ->> \@O(Y)$.
As the $\subseteq$ notation suggests, we identify open sets $U \in \@O(X)$ with open sublocales; this entails that we identify a further open $V \in \@O(U)$ with the same $V \subseteq U \in \@O(X)$, i.e., we identify $\@O(U)$ with $\down U \subseteq \@O(X)$ (rather than the fixed set of the corresponding nucleus, as in \cite{Jstone}).

We will need the notion of \defn{open map} $f : X -> Y$ between locales; see \cite[Ch.~V]{JTgpd} or \cite[C3.1]{Jeleph}.
A quick definition is that $f^* : \@O(Y) -> \@O(X)$ has a left adjoint, which we denote simply by $U |-> f(U)$, and this left adjoint is pullback-stable, in the sense that for any pullback square
\begin{equation*}
\begin{tikzcd}
Z \times_Y X \dar["f'"'] \rar["g'"] & X \dar["f"] \\
Z \rar["g"'] & Y
\end{tikzcd}
\end{equation*}
we have $f'(g^{\prime*}(U)) = g^*(f(U))$ (which implies, more generally, $f'(f^{\prime*}(V) \cap g^{\prime*}(U)) = V \cap g^*(f(U))$).
A basic example is given by a product projection $\pi_1 : X \times Y -> X$ (being the pullback of $Y -> \*1$).

Regarding more general \defn{images} of locale maps $f : X -> Y$, the unadorned term usually refers in the literature (see e.g., \cite[discussion before C1.2.5]{Jeleph}) to epi--regular mono factorization in $\!{Loc}$, i.e., the smallest sublocale of $Y$ through which $f$ factors.
Likewise, $f$ is usually called \defn{surjective} if it is an epimorphism in $\!{Loc}$.
\emph{However, these notions are too weak for our purposes, and we will only ever use their pullback-stable versions}, which we will usually name explicitly for emphasis.
(See \cref{thm:epi-pullback} below for a more ``intrinsic'' characterization of pullback-stable images.)

A sublocale $A \subseteq X$ is \defn{dense}, or more generally, a continuous map $f : A -> X$ has \defn{dense image}, if every nonempty open set in $X$ has nonempty pullback to $A$.
As mentioned in the Introduction, we have the following ``Baire category theorem'' for locales; see \cite[II~2.4]{Jstone}, \cite[C1.2.6(c)]{Jeleph}:

\begin{theorem}[Isbell]
\label{thm:bct}
In any locale $X$, the intersection of all dense sublocales is dense, i.e., there is a smallest dense sublocale (which also happens to be the intersection of all dense open sublocales).
\end{theorem}

We will say that a sublocale $A \subseteq X$ is \defn{fiberwise dense}, or more generally a locale map $f : A -> X$ has \defn{fiberwise dense image}, over another map $p : X -> Y$ if it is dense, and remains so after pulling back along any $Z -> Y$.
Thus, $A \subseteq X$ is dense iff it is fiberwise dense over $\*1$, i.e., for any $Z$, $A \times Z \subseteq X \times Z$ is dense.
Again, this definition is the pullback-stable strengthening of Johnstone's in \cite{Jgpd}, but is equivalent to it in the case of an open $p$, by \cite[2.7]{Jgpd}.
(See \cref{thm:dense-pullback} below for an ``intrinsic'' justification of our definition.)
The above ``Baire category theorem'' immediately implies its fiberwise generalization:

\begin{corollary}
Fiberwise dense sublocales (over a fixed base) are closed under arbitrary intersection.
\end{corollary}

\subsection{Localic groups}

A \defn{localic group} consists of an underlying locale $G$, equipped with continuous maps $m : G \times G -> G$ (multiplication), $i : G -> G$ (inverse), and $e \in G$, i.e., $e : \*1 -> G$ (identity), satisfying the group axioms internally in $\!{Loc}$.
For a basic reference, see \cite{Wgrp} or \cite[Ch.~XV]{PPloc}.

The following simple fact plays a key role in the theory of localic groups (cf.\ \cite[XV~4.1]{PPloc}):

\begin{lemma}
\label{thm:action-twist}
For any group object $G$ in a category with finite products, group action $a : G \times X -> X$ on another object, and morphism $f : A -> G$, the morphism $a \circ (f \times 1_X) : A \times X -> X$ is isomorphic over $X$ to the second projection $\pi_2$:
\begin{equation*}
\begin{tikzcd}[baseline=(\tikzcdmatrixname-2-1.base)]
A \times X
    \dar["f \times 1_X"']
    \rar[shift left, "{(\pi_1, a \circ (f \times 1_X))}"]
&[4em]
A \times X
    \dar["\pi_2"]
    \lar[shift left, "{(\pi_1, a \circ (if \times 1_X))}"]
\\
G \times X
    \rar[r, "a"'] &
X
\end{tikzcd}
\qed
\end{equation*}
\end{lemma}

\begin{corollary}
\label{thm:action-open}
For a localic group $G$, continuous action $a : G \times X -> X$, and continuous map $f : A -> G$, the map $a \circ (f \times 1_X) : A \times X -> X$ is open, i.e., for any $U \in \@O(A)$ and $V \in \@O(X)$, we have a pullback-stable open image
\begin{equation*}
f(U) \cdot V := \im(U \times V `-> A \times X --->{f \times 1_X} G \times X --->{a} X) \in \@O(X).
\qed
\end{equation*}
\end{corollary}

\begin{remark}
Taking $a := m$ and $f$ an embedding yields the main result of \cite{Wgrp}.
\end{remark}

Taking $f := 1_G$, we get that $m$ is open, i.e., for $U, V \in \@O(G)$, we have a product set $U \cdot V := m(U \times V) \in \@O(G)$.
For more general products of sets $S \cdot T := m(S \times T)$, where $S, T \subseteq G$ are not necessarily open (sublocales, say), our convention on images from above applies: we will only ever use the notion when the image is pullback-stable.
(See also \cref{thm:action-borel-open} below.)
We use similar notation for inverses of sets: $S^{-1} := i(S) = i^*(S)$ (using that $i = i^{-1}$).

Since the rectangles $U \times V$ form a basis of open sets in $G \times G$, we have, for any $W \in \@O(G)$,
\begin{equation*}
\textstyle
m^*(W) = \bigcup \{U \times V \mid U, V \in \@O(G) \AND U \cdot V \subseteq W\}.
\end{equation*}
Using this, the group axioms yield the following algebraic identities on open sets:
\begin{align}
\label{eq:grp-unit}
&U = \bigcup \{V \in \@O(G) \mid \exists e \in W \in \@O(G)\, (V \cdot W \subseteq U)\} &&\text{(by ``$g \cdot e = g$'')}, \\
\label{eq:grp-linv}
&\bigcup \underbrace{\{V \in \@O(G) \mid V^{-1} \cdot V \subseteq U\}}_{=: \@L_U}
= \begin{cases}
   G &\text{if $e \in U \in \@O(G)$}, \\
   \emptyset &\text{if $e \not\in U \in \@O(G)$}
   \end{cases}
&&\text{(by ``$g^{-1} \cdot g = e$'')}, \\
\label{eq:grp-rinv}
&\bigcup \underbrace{\{V \in \@O(G) \mid V \cdot V^{-1} \subseteq U\}}_{=: \@R_U}
= \text{similarly}.
\end{align}

For each $e \in U \in \@O(G)$, the set $\@L_U \subseteq \@O(G)$ on the left-hand side of \eqref{eq:grp-linv} thus forms an open cover of $G$; similarly for the set $\@R_U$ on the left-hand side of \eqref{eq:grp-rinv}.
Now if $e \in W \in \@O(G)$ and $V \cdot W \subseteq U$ as in \eqref{eq:grp-unit}, then for every $X \in \@L_W$ which meets $V$, we have
\begin{equation*}
\begin{aligned}
X
&= e \cdot X \\
&\subseteq (V \cap X) \cdot (V \cap X)^{-1} \cdot X  &&\text{by \eqref{eq:grp-rinv}, since $V \cap X \ne \emptyset$} \\
&\subseteq V \cdot X^{-1} \cdot X \\
&\subseteq V \cdot W \subseteq U;
\end{aligned}
\end{equation*}
thus since $\@L_W$ is an open cover, we get $V \subseteq \neg \bigcup \{X \in \@L_W \mid V \cap X = \emptyset\} \subseteq U$, which by \eqref{eq:grp-unit} shows that $G$ is \defn{regular} (every open set is the union of the interiors of its closed subsets).
In fact, the families of open covers $\@L_U$ and $\@R_U$, as $U$ varies over all identity neighborhoods, form bases for uniformities on $G$, called the \defn{left and right uniformities} on $G$ respectively; their join is the \defn{two-sided uniformity} (see \cite{BVgrp} for details; we will not need these uniformities, except briefly in passing, in \cref{rmk:polish}).

\subsection{Descriptive set theory}
\label{sec:dst}

After \cref{sec:pettis}, we will need to refer to some less standard descriptive set-theoretic notions from \cite{Cborloc}.
We now briefly review these, and use them to justify the pullback-stable notions defined in \cref{sec:loc}.
The reader who is interested only in the Pettis--Johnstone theorem can safely skip most of this subsection; the only thing from here needed in \cref{sec:pettis} is \cref{thm:dense-epi}, which uses only the standard notion of dissolution locale.

First, let us point out that our ``descriptive set theory'' differs from some other versions in the literature, e.g., \cite{Idst}, which develop everything inside the non-Boolean lattice of sublocales.
By contrast, our theory is fundamentally Boolean, and involves much more complicated ``sets'' than just sublocales; as a result, it closely resembles classical descriptive set theory (for Polish spaces).

Given a locale $X$, we let $\@B_\infty(X)$ denote the free complete Boolean algebra generated by the frame $\@O(X)$, and call its elements \defn{$\infty$-Borel sets} in $X$.
(Note that $\@B_\infty(X)$ is typically a proper class; see e.g., \cite[3.5.13]{Cborloc}.)
We may ramify $\@B_\infty(X)$ into the transfinite hierarchy given by the ``frame of nuclei'' functor $\@N$, which freely adjoins complements to frames (see e.g., \cite[C1.1.20]{Jeleph}):
\begin{equation*}
\underset{\substack{\parallel \\ \infty\Sigma^0_1(X)}}{\@O(X)} \subseteq
\underset{\substack{\parallel \\ \infty\Sigma^0_2(X)}}{\@N(\@O(X))} \subseteq
\underset{\substack{\parallel \\ \infty\Sigma^0_3(X)}}{\@N^2(\@O(X))} \subseteq
\dotsb \subseteq
\underset{\substack{\parallel \\ \infty\Sigma^0_\omega(X)}}{\@N^\omega(\@O(X))} \subseteq \dotsb \subseteq \@B_\infty(X)
\end{equation*}
We call the elements of $\@N^\alpha(\@O(X))$ the \defn{$\infty\Sigma^0_{1+\alpha}$-sets of $X$}, and their complements (in $\@B_\infty(X)$) the \defn{$\infty\Pi^0_{1+\alpha}$-sets}.
By a result of Isbell, the $\infty\Pi^0_2$-sets are in canonical order-preserving bijection with sublocales (in a manner compatible with the identification of open sets with sublocales).
We may thus use the notation $B \subseteq X$ to refer to any $\infty$-Borel set in $X$, compatible with the earlier use for sublocales.
See \cite[\S3]{Cborloc} for details.

Isbell's correspondence also justifies our calling \cref{thm:bct} a ``Baire category theorem'' for locales: it says that dense $\infty\Pi^0_2$-sets are closed under intersection.
We denote the smallest dense $\infty\Pi^0_2$-set (i.e., sublocale) of $X$ by
\begin{equation*}
\Comgr(X) \subseteq X,
\end{equation*}
and think of it as the smallest ``comeager'' set; we thus define an arbitrary $B \in \@B_\infty(X)$ to be \defn{comeager} if it contains $\Comgr(X)$ (equivalently, any intersection of dense $\infty\Pi^0_2$-sets), and \defn{meager} if disjoint from $\Comgr(X)$.
Recall that the quotient frame $\@O(X) ->> \@O(\Comgr(X))$ is given up to isomorphism by the \defn{regular open algebra} of $X$.
This easily implies that every $B \in \@B_\infty(X)$ has the \defn{property of Baire}, i.e., differs from an open set by a meager set.
See \cite[\S3.8]{Cborloc}.

For a continuous locale map $f : X -> Y$, the frame homomorphism $f^* : \@O(Y) -> \@O(X)$ extends to a complete Boolean homomorphism $f^* : \@B_\infty(Y) -> \@B_\infty(X)$.
More generally, we say that an \defn{$\infty$-Borel map} $f : X -> Y$ is an arbitrary complete Boolean homomorphism $f^* : \@B_\infty(Y) -> \@B_\infty(X)$ (equivalently, frame homomorphism $f^* : \@O(Y) -> \@B_\infty(X)$).
We let $\!{\infty BorLoc}$ denote the category of locales and $\infty$-Borel maps; then the forgetful functor $\!{Loc} -> \!{\infty BorLoc}$ is faithful and preserves limits.
Again see \cite[\S3]{Cborloc}.
In particular, every localic group $G$ has an underlying ``$\infty$-Borel group''; and it makes sense to ask whether an $\infty$-Borel set $B \in \@B_\infty(G)$ is a subgroup (e.g., closure under multiplication means $B \times B \subseteq m^*(B) \in \@B_\infty(G \times G)$).

\begin{lemma}[cf.\ {\cite[XV~2.2]{PPloc}}]
\label{thm:closure}
Let $G$ be a localic group, $B \in \@B_\infty(G)$ be an $\infty$-Borel subgroup.
Then the closure $\-B \in \infty\Pi^0_1(G)$ is still a (localic) subgroup.
\end{lemma}
\begin{proof}
Closure under inversion $i$ is straightforward.
To show closure of $\-B$ under multiplication $m$: we have $\-B \times \-B = \-{B \times B}$, because any basic open rectangle $U \times V$ disjoint from $B \times B$ must have $U \cap B = \emptyset$ or $V \cap B = \emptyset$, whence $ U \cap \-B = \emptyset$ or $V \cap \-B = \emptyset$, whence $(U \times V) \cap (\-B \times \-B) = \emptyset$; $\-{B \times B} \subseteq \-{m^*(B)}$ because $B$ is closed under $m$; and $\-{m^*(B)} \subseteq m^*(\-B)$ by continuity of $m$.
\end{proof}

The \defn{$\infty$-Borel image} of an $\infty$-Borel set $B \in \@B_\infty(X)$ under an $\infty$-Borel map $f : X -> Y$ is the smallest $f(B) \in \@B_\infty(Y)$ whose pullback contains $B$; the $\infty$-Borel image of the map $f$ is the $\infty$-Borel image of $X$ under it.
See \cite[\S3.4]{Cborloc} for a detailed discussion of this notion.
The $\infty$-Borel image may not exist \cite[4.4.3]{Cborloc}; but if it does exist, it is automatically pullback-stable.
The connection with image sublocales, and epimorphisms in $\!{Loc}$, is given by (see \cite[3.4.19]{Cborloc}):

\begin{proposition}[Wilson]
\label{thm:epi-pullback}
A continuous map $f : X -> Y \in \!{Loc}$ is a pullback-stable epimorphism in $\!{Loc}$ iff it is \defn{$\infty$-Borel surjective}, i.e., $Y$ is its $\infty$-Borel image.

Thus, the epi--regular mono image of $f : X -> Y$, i.e., the smallest sublocale (= $\infty\Pi^0_2$-set) of $Y$ whose pullback is all of $X$, is the $\infty$-Borel image iff it is a pullback-stable factorization in $\!{Loc}$.
\end{proposition}

In particular, this means that open maps, defined the usual way (see \cref{sec:loc}), are equivalently defined by each open set in the domain having an $\infty$-Borel image which is open.

The proof of \cref{thm:epi-pullback}, and similar results connecting $\infty$-Borel notions with pullback-stability, is based on the fact that every $\infty$-Borel set $B \in \@B_\infty(X)$ can be made open after ``refining the topology of $X$''.
The \defn{dissolution} of $X$ is the locale $\@D(X)$ with $\@O(\@D(X)) := \@N(\@O(X))$ (see \cite[p498]{Jeleph}, \cite[\S3.3]{Cborloc}).
Thus, every $\infty$-Borel set becomes open in a sufficiently high dissolution $\@D^\alpha(X) -> X$; and the $\infty$-Borel category $\!{\infty BorLoc}$ can be defined as the localization of $\!{Loc}$ inverting all such dissolution maps.
In particular, every $\infty$-Borel set $B$ is the $\infty$-Borel image of a continuous monomorphism (namely the open sublocale $B \subseteq \@D^\alpha(X)$).
Taking such an $f$ in \cref{thm:action-open} shows

\begin{corollary}
\label{thm:action-borel-open}
For a localic group $G$, continuous action $a : G \times X -> X$, $\infty$-Borel $B \subseteq G$, and open $U \subseteq X$, the $\infty$-Borel image $B \cdot U := a(B \times U) \subseteq X$ exists and is open.
\qed
\end{corollary}

The definitions of an $\infty$-Borel set $B \subseteq X$ being \defn{(fiberwise) dense} (over continuous $p : X -> Y$), or more generally, of an $\infty$-Borel map $f : B -> X$ having \defn{(fiberwise) dense image}, are the same as in \cref{sec:loc}.
Again by considering dissolutions of $B$, this notion actually reduces to the earlier notion of a continuous $f$ having (fiberwise) dense image.
For fiberwise density, considering dissolutions again yields the following more ``intrinsic'' characterization, which will however not be needed in the rest of the paper;
its proof uses a result from \cite{Cborloc} as a black box.

\begin{proposition}
\label{thm:dense-pullback}
For $\infty$-Borel $f : A -> X$ and continuous $p : X -> Y$, the following are equivalent:
\begin{enumerate}
\item[(i)]  $f$ has fiberwise dense image over $p$, in the pullback-stable sense defined in \cref{sec:loc};
\item[(ii)]  the pullback under $f$ of every nonempty ``basic fiberwise open'' set in $X$, of the form $U \cap p^*(B)$ where $U \in \@O(X)$ and $B \in \@B_\infty(Y)$, remains nonempty.
\end{enumerate}
\end{proposition}

\begin{remark}
By the classical Kunugui--Novikov uniformization theorem for fiberwise open Borel sets \cite[28.7]{Kcdst}, a reasonable definition of ``fiberwise open $\infty$-Borel set'' in $X$ would be arbitrary unions of sets of the form $U \cap p^*(B)$ as in (ii).
\end{remark}

\begin{proof}[Proof of \cref{thm:dense-pullback}]
If (i) holds, then (ii) follows by pulling back to a sufficiently high dissolution $\@D^\alpha(Y) -> Y$ making $B$ (hence $U \cap p^*(B)$) open.
Conversely, suppose (i) fails; let
\begin{equation*}
\begin{tikzcd}
Z \times_Y A \dar["f'"'] \rar["g''"] & A \dar["f"] \\
Z \times_Y X \dar["p'"'] \rar["g'"] & X \dar["p"] \\
Z \rar["g"] & Y
\end{tikzcd}
\end{equation*}
be a pullback such that $f'$ no longer has dense image.
By definition of the frame tensor product $\@O(Z \times_Y X) = \@O(Z) \otimes_{\@O(Y)} \@O(X)$, this means there are open $U \subseteq X$ and $V \subseteq Z$ such that
\begin{equation*}
\emptyset \ne p^{\prime*}(V) \cap g^{\prime*}(U), \qquad
\emptyset = f^{\prime*}(p^{\prime*}(V) \cap g^{\prime*}(U))
= f^{\prime*}(p^{\prime*}(V)) \cap g^{\prime\prime*}(f^*(U)).
\end{equation*}
By the Lusin separation theorem for locales \cite[4.2.1; see also 2.12.6]{Cborloc}, there is an $\infty$-Borel $B \subseteq Y$ so that
\begin{equation*}
V \subseteq g^*(B), \qquad
\emptyset = f^*(U) \cap f^*(p^*(B)) = f^*(U \cap p^*(B));
\end{equation*}
the former implies
\begin{equation*}
\emptyset
\ne p^{\prime*}(V) \cap g^{\prime*}(U)
\subseteq p^{\prime*}(g^*(B)) \cap g^{\prime*}(U)
= g^{\prime*}(p^*(B) \cap U),
\end{equation*}
whence $U \cap p^*(B) \ne \emptyset$, whence (ii) fails.
\end{proof}

Finally, we have the following connection between fiberwise density and $\infty$-Borel (i.e., pullback-stable) images.
We again first state the continuous version, which is the pullback-stable version of \cite[1.11(i)]{Jgpd}; the $\infty$-Borel version again follows by considering sufficiently high dissolutions.

\begin{lemma}
\label{thm:dense-epi}
If a continuous $f : A -> X$ has fiberwise dense image over $p : X -> Y$ (in the pullback-stable sense of \cref{sec:loc}), then $p$ is a pullback-stable epimorphism in $\!{Loc}$ iff $p \circ f$ is.
\end{lemma}
\begin{proof}
$\Longleftarrow$ is clear.
Conversely, if $p$ is a pullback-stable epimorphism, then for any sublocale $B \subseteq Y$ such that $f^*(p^*(B)) = A$, after pulling back to the dissolution $\@D(Y) -> Y$, $B$ becomes closed, whence so does $p^*(B)$, whence fiberwise density of $f$ implies that $p^*(B) = X$, whence $B = Y$ since the pullback of $p$ is an epimorphism; this shows that $p \circ f$ is an epimorphism.
Now running the same argument after first pulling back along some $Z -> Y$ shows that $p \circ f$ is pullback-stable.
\end{proof}

\begin{corollary}
If an $\infty$-Borel $f : B -> X$ has fiberwise dense image over continuous $p : X -> Y$, then $f$ and $p \circ f$ have the same $\infty$-Borel image, either existing iff the other does.
\qed
\end{corollary}

\section{Pettis's theorem}
\label{sec:pettis}

We recall Pettis's theorem in the following form (the version for difference sets cited in the Introduction follows by applying the property of Baire and taking $T = S^{-1}$):

\begin{theorem}[Pettis]
Let $G$ be a Baire (i.e., satisfying the Baire category theorem) topological group, let $U, V \subseteq G$ be open, and let $S \subseteq U$ and $T \subseteq V$ be dense $G_\delta$.
Then $U \cdot V = S \cdot T$.
\end{theorem}
\begin{proof}
For any $g \in G$, we have
\begin{equation*}
g \in U \cdot V
\iff U \cap (g \cdot V^{-1}) \ne \emptyset
\iff S \cap (g \cdot T^{-1}) \ne \emptyset
\iff g \in S \cdot T,
\end{equation*}
where the middle $\Longrightarrow$ is because $S \subseteq U$ and $T \subseteq V$ are comeager, hence $S \cap (g \cdot T^{-1}) \subseteq U \cap (g \cdot V^{-1})$ is also comeager, therefore dense.
\end{proof}

In order to make this proof point-free, we need to work over all $g$ at once.
So consider the sets
\begin{alignat*}{2}
Q &:= \{(h, g) \in G \times G \mid h \in U \cap (g \cdot V^{-1})\} &&= \{(h, g) \mid h \in U \AND h^{-1} \cdot g \in V\}, \\
P &:= \{(h, g) \in G \times G \mid h \in S \cap (g \cdot T^{-1})\} &&= \{(h, g) \mid h \in S \AND h^{-1} \cdot g \in T\},
\end{alignat*}
whose projections onto the second coordinate yield $U \cdot V$ and $S \cdot T$ respectively: indeed, we have
\begin{align*}
U \cdot V
= m(U \times V)
= m((U \times G) \cap (G \times V))
&= \pi_2(\underbrace{(U \times G) \cap (\pi_1, m \circ (i \times 1_G))^*(G \times V)}_Q) \\
\intertext{using the ``twist'' isomorphism between $m$ and $\pi_2$ from \cref{thm:action-twist}.
Now note that we may rewrite this, using
$(\pi_1, m \circ (i \times 1_G))^*(G \times V)
= (m \circ (i \times 1_G))^*(V)
= (m \circ (i \times 1_G), \pi_2)^*(V \times G)$:}
&= \pi_2((U \times G) \cap (m \circ (i \times 1_G), \pi_2)^*(V \times G));
\end{align*}
moreover, this latter map $(m \circ (i \times 1_G), \pi_2) : G \times G -> G \times G$ (taking $(g, h) |-> (g^{-1}h, h)$) is \emph{also} a ``twist'' isomorphism, this time preserving $\pi_2$.
Of course, a similar description applies to $P$:
\begin{equation*}
S \cdot T
= m(S \times T)
= \pi_2(P)
= \pi_2((S \times G) \cap (m \circ (i \times 1_G), \pi_2)^*(T \times G)).
\end{equation*}
Since $S \subseteq U$ and $T \subseteq V$ are comeager, $S \times G \subseteq U \times G$ and $T \times G \subseteq V \times G$ are fiberwise comeager (over $\pi_2$), and the latter remains so after the ``twist''; hence their intersection $P \subseteq Q$ is fiberwise comeager as well, and so
\begin{equation*}
U \cdot V = \pi_2(Q) = \pi_2(P) = S \cdot T.
\end{equation*}
This argument works verbatim for localic groups (using \cref{thm:dense-epi} for the last step), yielding

\begin{theorem}[Pettis--Johnstone]
\label{thm:johnstone}
Let $G$ be a localic group, let $U, V \subseteq G$ be open, and let $S \subseteq U$ and $T \subseteq V$ be dense sublocales (e.g., the smallest such).
Then $U \cdot V = S \cdot T$, in the sense that
\begin{align*}
m : S \times T --> U \cdot V
\end{align*}
is a pullback-stable locale epimorphism.
\qed
\end{theorem}

We have written the above proof so that (once interpreted in the localic context) it is nearly identical to that of \cite[3.1]{Jgpd}.
The only differences are that we consider dense sublocales of two arbitrary open $U, V \subseteq G$ (instead of $U = V = G$), and that \cite[3.1]{Jgpd} does not explicitly state pullback-stability.

\begin{remark}
\label{rmk:johnstone}
We have stated \cref{thm:johnstone} without reference to the $\infty$-Borel notions from \cref{sec:dst}; but the statement clearly implies more generally that for two comeager $\infty$-Borel $S \subseteq U$ and $T \subseteq V$ (in the sense of \cref{sec:dst}), we have $U \cdot V = S \cdot T$ in the sense of an $\infty$-Borel image.
\end{remark}

\section{Consequences for group homomorphisms}
\label{sec:homom}

As mentioned in the Introduction, Johnstone \cite{Jgpd} used \cref{thm:johnstone} to show the following result of Isbell--Kříž--Pultr--Rosický \cite[4.7]{IKPRgrp}:

\begin{corollary}[closed subgroup theorem]
\label{thm:closed}
Every localic subgroup of a localic group is closed.
\end{corollary}
\begin{proof}
If $G$ is a localic group, and $H \subseteq G$ is a localic subgroup, then $H$ is a dense sublocale of its closure $\-H$ which is also a localic subgroup
(\cref{thm:closure}), whence $\-H = \-H \cdot \-H = H \cdot H = H$.
\end{proof}

We now derive the localic analogs of some other classical results:

\begin{theorem}[automatic continuity]
\label{thm:cts}
Every $\infty$-Borel group homomorphism $f : G -> H$ between localic groups is continuous.
\end{theorem}
\begin{proof}
First, we show that for every $U \in \@O(H)$ containing the identity element $e_H$, the $\infty$-Borel set $f^*(U) \in \@B_\infty(G)$ is a neighborhood of $e_G$, i.e., contains an open set containing $e_G$.
From \eqref{eq:grp-rinv},
\begin{align*}
G &= \bigcup \{f^*(V) \mid V \in \@O(H) \AND V \cdot V^{-1} \subseteq U\}.
\end{align*}
By the Baire category theorem \ref{thm:bct}, some $f^*(V)$ must be non-meager, hence by the property of Baire (see \cref{sec:dst}), there is $\emptyset \ne W \in \@O(G)$ such that $W \cap f^*(V) \subseteq W$ is comeager.
Since $W \ne \emptyset$, we have $e_G \in W \cdot W^{-1}$ again by \eqref{eq:grp-rinv}.
By \cref{thm:johnstone} (or rather, \cref{rmk:johnstone}), morally speaking, we now have
\begin{equation*}
e_G \in W \cdot W^{-1}
= (W \cap f^*(V)) \cdot (W \cap f^*(V))^{-1}
\subseteq f^*(V) \cdot f^*(V)^{-1}
\subseteq f^*(V \cdot V^{-1})
\subseteq f^*(U).
\end{equation*}
The only wrinkle is that the $\infty$-Borel image $f^*(V) \cdot f^*(V)^{-1} \subseteq G$ may not exist.
It is possible to work around this issue using the more general notion of ``$\infty$-analytic set'' (the result of formally adjoining missing images to $\!{\infty BorLoc}$) from \cite{Cborloc}.
If we unravel the resulting argument, we get the following, which stays in the $\infty$-Borel realm: to say that
\begin{equation*}
(W \cap f^*(V)) \cdot (W \cap f^*(V))^{-1} \subseteq f^*(V \cdot V^{-1})
\end{equation*}
means that the multiplication $m : G \times G -> G$ obeys
\begin{equation*}
(W \cap f^*(V)) \times (W \cap f^*(V))^{-1} \subseteq m^*(f^*(V \cdot V^{-1}));
\end{equation*}
but the left-hand side is contained in $f^*(V) \times f^*(V^{-1})$ (using that $f$ preserves inverse), which is contained in $m^*(f^*(V \cdot V^{-1}))$ by pullback-stability of the image $V \cdot V^{-1}$ (along the $\infty$-Borel $f$).

To complete the proof, let $U \in \@O(H)$; we must show that $f^*(U) \subseteq G$ is open.
From \eqref{eq:grp-unit},
\begin{align*}
f^*(U)
&= \bigcup \{f^*(V) \mid V \in \@O(H) \AND \exists e_H \in W \in \@O(H)\, (V \cdot W \subseteq U)\}.
\end{align*}
For each such $V, W \in \@O(H)$ with $e_H \in W$ and $V \cdot W \subseteq U$, we have shown above that $f^*(W)$ contains an open $\~W$ containing $e_G$, whence
\begin{equation*}
f^*(V)
= f^*(V) \cdot e_G
\subseteq f^*(V) \cdot \~W
\subseteq f^*(V \cdot W)
\end{equation*}
where the middle set $f^*(V) \cdot \~W$ is open by \cref{thm:action-borel-open}
(and the last $\subseteq$ is by pullback-stability of $V \cdot W$, similarly to above).
But the unions of the left and right sets, over all such $V, W$, are both $f^*(U)$, which is hence open.
\end{proof}

\begin{corollary}
Every localic group monomorphism $f : G -> H$ which is also a pullback-stable epimorphism of locales (i.e., an $\infty$-Borel surjection) is a localic group isomorphism.
\end{corollary}
\begin{proof}
By the Lusin--Suslin theorem for locales (i.e., LaGrange's theorem \cite{Lamalg} that epimorphisms of complete Boolean algebras are surjective; see \cite[3.4.25]{Cborloc} for an explanation of its descriptive set-theoretic meaning), $f$ has an $\infty$-Borel inverse, which is automatically continuous.
\end{proof}

Now consider an arbitrary localic group homomorphism $f : G -> H$.
By regularity of $H$ (or by the closed subgroup theorem \ref{thm:closed}), the singleton $\{e_H\} \subseteq H$ is closed, whence so is its preimage
\begin{equation*}
\ker(f) = f^*(e_H) \subseteq G,
\end{equation*}
as is the kernel pair of $f$
\begin{equation*}
{\sim_f} = (m \circ (i \times 1_G))^*(\ker(f)) \subseteq G^2.
\end{equation*}
Moreover, the projections $\pi_1, \pi_2 : G^2 \rightrightarrows G$ restricted to $\sim_f$ are open maps; indeed, for a basic open rectangle $U \times V \subseteq G^2$ where $U, V \in \@O(G)$, we have
\begin{align*}
\pi_2({\sim_f} \cap (U \times V))
&= \pi_2((\pi_1, m \circ (i \times 1_G))^*(G \times \ker(f)) \cap (U \times V)) \\
&= m((G \times \ker(f)) \cap (\pi_1, m)^*(U \times V)) \qquad\text{by \cref{thm:action-twist}} \\
&= m((U \times \ker(f)) \cap m^*(V))
\end{align*}
which is the open set $(U \cdot \ker(f)) \cap V \subseteq G$ by \cref{thm:action-open} and pullback-stability along $V `-> G$.
It now follows from the descent theory of open surjections (see \cite[C5.1.4, C5.1.9]{Jeleph}) that the coequalizer of ${\sim_f} \rightrightarrows G$ is an open surjection with kernel pair $\sim_f$, through which $f$ factors as a monomorphism (because its kernel pair must pull back to $\sim_f$), yielding a factorization
\begin{equation*}
f : G ->> G/\ker(f) `-> H
\end{equation*}
of an arbitrary localic group homomorphism as an open surjection followed by a monomorphism; and $G/\ker(f)$ is easily seen to inherit a localic group structure from $G$ (using that $\sim_f$ is a group congruence), making both maps in this factorization group homomorphisms.

\begin{corollary}[open mapping theorem]
\label{thm:open}
Every localic group homomorphism $f : G -> H$ which is a pullback-stable epimorphism of locales (i.e., an $\infty$-Borel surjection) is open.
\end{corollary}
\begin{proof}
The monic part of the above factorization is also a pullback-stable epimorphism of locales, hence an isomorphism by the preceding corollary.
\end{proof}

\begin{remark}
The above factorization shows that the category $\!{LocGrp}$ of localic groups is regular.
It is not Barr-exact, however, since a localic group may have a non-embedded normal subgroup (with a ``finer topology'', e.g., $\#Q `-> \#R$ where $\#Q$ is discrete), which corresponds to an internal congruence in $\!{LocGrp}$ which is not a kernel pair.
The exact completion of $\!{LocGrp}$ has recently been considered, in the second-countable case (see \cref{rmk:polish}), by Bergfalk--Lupini--Panagiotopoulos \cite{BLPhom}, as a suitable context for ``definable'' homological algebra.
\end{remark}

\begin{remark}
In contrast to \cref{thm:open}, of course, a pullback-stable epimorphism of locales (e.g., any epimorphism between compact Hausdorff locales; see \cite[C3.2]{Jeleph}) need not be open.
Likewise, there are continuous locale maps which do not have a pullback-stable epi--mono factorization, i.e., $\infty$-Borel image (again see \cite[4.4.3]{Cborloc}), let alone a regular factorization.
Thus the category $\!{LocGrp}$ is much better behaved than $\!{Loc}$.
\end{remark}

We end this section by briefly discussing the precise connection between the above results and their classical analogs.
Recall that a \defn{Polish group} is a second-countable topological group which is completely metrizable, or equivalently, complete in its two-sided uniformity (see \cite[\S2.2]{Gidst}).

\begin{remark}
\label{rmk:polish}
Polish groups form a category equivalent to second-countable localic groups.
\end{remark}
\begin{proof}[Proof sketch]
On the one hand, because a Polish group $G$ is completely metrizable, its localic product $G \times G$ agrees with the topological product, whence the group operations pass to the underlying locale of $G$ (see e.g., \cite[6.1]{IKPRgrp}).
On the other hand, by Banaschewski--Vermeulen \cite{BVgrp}, every localic group $G$ is complete in its two-sided uniformity; if $G$ is moreover first-countable, i.e., the identity $e \in G$ has a countable neighborhood basis, then the two-sided uniformity is countably generated, hence metrizable, whence $G$ is spatial, hence is (isomorphic to the underlying localic group of) a two-sided complete topological group.
(See \cite[X~2.2]{PPloc} or \cite[3.2]{Iloc} for details; briefly, metrizability is by a point-free adaptation of the Birkhoff--Kakutani metrization theorem for first-countable topological groups \cite[2.1.1]{Gidst}, while spatiality is by intersecting nested sequences of closed sets of vanishing diameter in the usual manner.)
\end{proof}

The closed subgroup theorem \ref{thm:closed} thus generalizes the classical fact \cite[2.2.1]{Gidst} that every Polish (i.e., $G_\delta$) subgroup $H \subseteq G$ of a Polish group is closed.
The classical fact is more easily proved by observing that $H$ cannot be disjoint from any of its cosets in $\-H$, by Baire category; but it can also be proved via Pettis's theorem, exactly as in \cref{thm:closed}, which is needed in the point-free context because there need not be any single non-identity coset of $H$ (if $G/H$ is non-spatial).

\Cref{thm:cts} generalizes the classical automatic continuity of Borel homomorphisms between Polish groups, while \cref{thm:open} generalizes the classical open mapping theorem for surjective homomorphisms between Polish groups \cite[2.3.3]{Gidst} (there are of course also classical versions for Fréchet vector spaces).
The former result is indeed a generalization, because for a Polish group $G$, the classical Borel $\sigma$-algebra $\@B(G)$ is a Boolean $\sigma$-subalgebra of the $\infty$-Borel algebra $\@B_\infty(G)$ as defined in \cref{sec:dst} (see e.g., \cite[3.5.6]{Cborloc}); thus $\infty$-Borel group homomorphisms between Polish groups include all classical Borel group homomorphisms.
For the open mapping theorem, the additional fact needed for the correspondence with the classical result is that open maps between Hausdorff spaces in the spatial and localic senses agree (see e.g., \cite[7.2.1(2)]{PPloc}).
Both results for localic groups are proved via direct point-free translations of the classical proofs in \cite[2.3.3]{Gidst}.

\section{Coproducts of complete Boolean algebras}
\label{sec:cbool}

The classical Gaifman--Hales theorem \cite{Gcbool}, \cite{Hcbool} (see also \cite{Scbool}) asserts that the free complete Boolean algebra on countably infinitely many generators does not exist.
In other words, if we construct the free algebra syntactically in the usual way, by taking equivalence classes of complete Boolean terms built from the generating set $\#N$, the resulting algebra is \emph{large}, i.e., a proper class.
Letting $\!{CBool}$ be the category of (small) complete Boolean algebras, the free algebra on $n$ generators is the same thing as the coproduct of $n$ copies of the free algebra $\{\top, a, \neg a, \bot\}$ on one generator; thus the Gaifman--Hales theorem says that the countable coproduct of this small algebra does not exist in $\!{CBool}$, in that it lands instead in the category $\!{CBOOL}$ of possibly large algebras.

On the other hand, the free complete Boolean algebra on a finite number $n$ of generators is still finite, coinciding with the free Boolean algebra on $n$ generators, namely $\@P(\@P(n))$.
It is thus natural to ask whether, more generally, every finite coproduct of small complete Boolean algebras is still small; to our knowledge, this question has never been addressed before.
The Pettis--Johnstone theorem yields a negative answer:

\begin{theorem}
\label{thm:cbool-coprod}
There exists a small complete Boolean algebra whose coproduct with itself (in the category of possibly large complete Boolean algebras) is large.
\end{theorem}
\begin{proof}
Let $G$ be any uncountable Polish group (e.g., $\#R$); we claim that the regular open algebra $\@O(\Comgr(G))$ has large coproduct with itself.
Since free functors preserve coproducts, the complete Boolean coproduct in question is the free complete Boolean algebra generated by the frame coproduct $\@O(\Comgr(G)) \otimes \@O(\Comgr(G)) = \@O(\Comgr(G) \times \Comgr(G))$, i.e., the $\infty$-Borel algebra $\@B_\infty(\Comgr(G) \times \Comgr(G))$.
By \cref{thm:johnstone}, the multiplication $m : \Comgr(G) \times \Comgr(G) -> G$ is $\infty$-Borel surjective, i.e., $m^* : \@B_\infty(G) -> \@B_\infty(\Comgr(G) \times \Comgr(G))$ is injective.
But since $G$, being uncountable, is Borel isomorphic to $2^\#N$, $\@B_\infty(G)$ is isomorphic to the free complete Boolean algebra on countably many generators (see \cite[3.5.13]{Cborloc}), hence large by Gaifman--Hales.
\end{proof}

\bigskip\noindent
Department of Mathematics and Statistics \\
CRM/McGill University \\
Montréal, QC, H3A 0B9, Canada \\
Email: \nolinkurl{ruiyuan.chen@umontreal.ca}


\begin{thebibliography}{0000000}

\bibitem[BV99]{BVgrp}  B.~Banaschewski and J.~J.~C.~Vermeulen, \emph{On the completeness of localic groups}, Comment.\ Math.\ Univ.\ Carolin.\ \textbf{40}~(1999), no.~2, 293--307.

\bibitem[BLP20]{BLPhom}  J.~Bergfalk, M.~Lupini, and A.~Panagiotopoulos, \emph{The definable content of homological invariants I: $\Ext$ \amp{} $\lim^1$}, preprint, \url{https://arxiv.org/abs/2008.08782}, 2020.

\bibitem[Che20]{Cborloc}  R.~Chen, \emph{Borel and analytic sets in locales}, preprint, \url{https://arxiv.org/abs/2011.00437v1}, 2020.

\bibitem[Gai64]{Gcbool}  H.~Gaifman, \emph{Infinite Boolean polynomials.\ I},  Fund.\ Math.\ \textbf{54}~(1964), 229--250.

\bibitem[Gao09]{Gidst}  S.~Gao, \emph{Invariant descriptive set theory}, Pure and Applied Mathematics (Boca Raton), vol.~293, CRC Press, Boca Raton, FL, 2009.

\bibitem[Hal64]{Hcbool}  A.~W.~Hales, \emph{On the non-existence of free complete Boolean algebras}, Fund.\ Math.\ \textbf{54}~(1964), 45--66.

\bibitem[Isb72]{Iloc}  J.~R.~Isbell, \emph{Atomless parts of spaces}, Math.\ Scand.\ \textbf{31}~(1972), 5--32.

\bibitem[Isb91]{Idst}  J.~R.~Isbell, \emph{First steps in descriptive theory of locales}, Trans.\ Amer.\ Math.\ Soc.\ \textbf{327}~(1991), no.~1, 353--371.

\bibitem[IKPR88]{IKPRgrp}  J.~Isbell, I.~Kříž, A.~Pultr, and J.~Rosický, \emph{Remarks on localic groups}, In: \emph{Categorical algebra and its applications}, Lecture Notes in Mathematics, vol.~1348, Springer-Verlag, Berlin, 1988.

\bibitem[Joh82]{Jstone}  P.~T.~Johnstone, \emph{Stone spaces}, Cambridge Studies in Advanced Mathematics, vol.~3, Cambridge University Press, Cambridge, England, 1982.

\bibitem[Joh89]{Jgpd}  P.~T.~Johnstone, \emph{A constructive ``closed subgroup theorem'' for localic groups and groupoids}, Cahiers Topologie Géom.\ Différentielle Catég.\ \textbf{30}~(1989), no.~1, 3--23.

\bibitem[Joh02]{Jeleph}  P.~T.~Johnstone, \emph{Sketches of an elephant: a topos theory compendium}, Oxford Logic Guides, vol.~43--44, Oxford University Press, Oxford, 2002.

\bibitem[JT84]{JTgpd}  A.~Joyal and M.~Tierney, \emph{An extension of the Galois theory of Grothendieck}, Mem.\ Amer.\ Math.\ Soc.\ \textbf{51}~(1984), no.~309, vii+71.

\bibitem[Kec95]{Kcdst}  A.~S.~Kechris, \emph{Classical descriptive set theory}, Graduate Texts in Mathematics, vol.~156, Springer-Verlag, 1995.

\bibitem[LaG74]{Lamalg}  R.~LaGrange, \emph{Amalgamation and epimorphisms in $\mathfrak{m}$ complete Boolean algebras}, Algebra Unversalis \textbf{4}~(1974), 277--279.

\bibitem[PP12]{PPloc}  J.~Picado and A.~Pultr, \emph{Frames and locales: Topology without points}, Frontiers in Mathematics, Birkhäuser/Springer Basel AG, Basel, 2012.

\bibitem[Sol66]{Scbool}  R.~M.~Solovay, \emph{New proof of a theorem of Gaifman and Hales}, Bull.\ Amer.\ Math.\ Soc.\ \textbf{72}~(1966), 282--284.

\bibitem[Wra90]{Wgrp}  G.~C.~Wraith, \emph{Unsurprising results on localic groups}, J.\ Pure Appl.\ Algebra \textbf{67}~(1990), no.~1, 95--100.

\end{thebibliography}
\end{document}